\documentclass[12pt]{article}
\usepackage[left=0.90in, right=0.90in,top=.75in, bottom=.75in]{geometry}

\usepackage{amssymb, amsfonts, amsmath, amsthm, enumerate, multicol, xspace, array, multirow, graphicx, fontenc, wasysym, verbatim,float}
\usepackage{proof}
\usepackage{tikz}

\newcommand{\bigdoublewedge}{%
  \mathop{
    \mathchoice{\bigwedge\mkern-15mu\bigwedge}
               {\bigwedge\mkern-12.5mu\bigwedge}
               {\bigwedge\mkern-12.5mu\bigwedge}
               {\bigwedge\mkern-11mu\bigwedge}
    }
}

\newcommand{\bigdoublevee}{%
  \mathop{
    \mathchoice{\bigvee\mkern-15mu\bigvee}
               {\bigvee\mkern-12.5mu\bigvee}
               {\bigvee\mkern-12.5mu\bigvee}
               {\bigvee\mkern-11mu\bigvee}
    }
}

\usetikzlibrary{decorations.pathreplacing,calc}

\def\len{\mathrm{len}}
\def\rtree{\mathrm{tree}}

\def\mbC{\mathbb{C}}

\def\me{\mathfrak{E}}
\def\ma{\mathfrak{A}}

\def\vp{\varphi}

\newcommand{\newmacro}[1]{\mathfrak{#1}}
\def\A{\mathbb{A}}
\newcommand{\E}{\newmacro{E}}
\newcommand{\vwE}{\bar{\E}}
\newcommand{\vwA}{\bar{\ma}}

\newcommand{\mc}[1]{\mathcal{#1}}

\newtheorem{lem}{Lemma}[section]

\newtheorem{prop}[lem]{Proposition}
\newtheorem{thm}[lem]{Theorem}
\newtheorem{remark}[lem]{Remark}
\newtheorem{defn}[lem]{Definition}

\newtheorem*{prop*}{Proposition}
\newtheorem*{thm*}{Theorem}
\newtheorem*{def*}{Definition}
\newtheorem*{lem*}{Lemma}

\title{Scott spectral gaps for trees are bounded}
\author{Matthew Harrison-Trainor and J. Thomas Kim\thanks{Harrison-Trainor was partially supported by a Sloan Research Fellowship and by the National Science Foundation under Grant DMS-2452105 and DMS-2419591/DMS-2153823. Thomas Kim was supported as an REU student by the National Science Foundation under grant DMS-2419591.}}

\begin{document}

\maketitle

\begin{abstract}
    Given a Borel class of trees, we show that there is a tree in that class whose Scott sentence is not too much more complicated than the definition of the class. In particular, if the class is definable by a $\Pi_\alpha$ sentence, then there is a model of Scott rank at most $\alpha + 2$. This gives another proof---and one that does not require first proving Vaught's conjecture for trees---of the fact that trees are not faithfully Borel complete.
\end{abstract}

\subsubsection*{AI Disclosure}

The first draft and initial submission of this paper was written without the use of any AI tools. During the refereeing process, the first author used ChatGPT Astra 6 to search for linguistic and mathematical errors which were then fixed.

\section{Introduction}

Given a countable mathematical structure $\mc{A}$, such as a graph, group, or linear order, Scott \cite{Sco65} showed that there is a sentence $\varphi$ of the infinitary logic $\mc{L}_{\omega_1 \omega}$ that characterizes that structure up to isomorphism among countable structures in the sense that for countable structures $\mc{B}$, we have $\mc{B} \models \varphi$ if and only if $\mc{B} \cong \mc{A}$. Such a sentence is called a \textit{Scott sentence} for $\mc{A}$. We can measure the complexity of describing $\mc{A}$ up to isomorphism by assigning to $\mc{A}$ an ordinal-valued Scott rank. There are a number of almost but not quite equivalent definitions of Scott rank. The one that we will use is due to Montalb\'an \cite{MonSR} and defines the \textit{Scott rank} of $\mc{A}$ to be the least $\alpha$ such that $\mc{A}$ has a $\Pi_{\alpha + 1}$ Scott sentence. This definition is particularly robust, and measures in addition the complexity of describing the automorphism orbits of the structure.

Given an $\mc{L}_{\omega_1 \omega}$ sentence $\varphi$, we think of $\varphi$ as a theory defining a class of models, and we ask what the possible Scott ranks of the models of $\varphi$ might be.

\begin{defn}
    Let $\varphi$ be an $\mc{L}_{\omega_1 \omega}$ sentence. The \emph{Scott spectrum} of $\varphi$ is the set
    \[ \text{SS}(\varphi) = \{ \alpha \in \omega_1 \; | \; \text{there is $\mc{A} \models \varphi$ of Scott rank $\alpha$}\}.\]
\end{defn}

\noindent In \cite{HTScott} the first author showed that there are theories $\varphi$ which have low complexity, but all of whose models have high complexity. For example, for any $\alpha$, there is a $\Pi_2$ sentence $\varphi$ all of whose models have Scott rank $\geq \alpha$. Moreover, such an example could be found within any universal class of structures (see \cite{HKSS} and Chapter VI of \cite{MBook1}), such as graphs or fields. On the other hand, in \cite{HTGonzalez}, the first author and Gonzalez showed that if a $\Pi_{\alpha}$  sentence $\varphi$ extends the axioms of linear orders then there is a model of $\varphi$ of Scott rank at most $\alpha +4$. This was the first non-trivial case in which such a result is known.

In this paper we consider the case of \textit{(simple) trees}, which are rooted connected acyclic graphs. We use the language $\mc{L}_{tree} = \{ r, P \}$ where $r$ is a constant standing for the root and $P$ is the parent relation, that is, $P(x) = y$ means that $y$ is the parent of $x$. (We write the parent relation using functional notation for clarity, but formally we take our language to be a relational language so that we can define the back-and-forth relations in the standard way.) Our motivation for considering trees arises from faithful Borel embeddings, a topic somewhat connected to Vaught's conjecture. Steel \cite{Ste78} proved that Vaught's conjecture is true for linear orders and trees,\footnote{And indeed for a more general definition of trees than the one that we consider here.} while of course Vaught's conjecture is open in general. While linear orders and trees are Borel complete \cite{FS89}, this is not enough to imply that Vaught's conjecture is true in general. The reason is that, as shown by Gao \cite{Gao01}, linear orders and trees are not faithfully Borel complete, that is, the saturation of the image of a Borel reduction from graphs to linear orders or trees cannot be Borel. For classes of structures which are faithfully Borel complete, including for example graphs, Vaught's conjecture for that class would imply the full Vaught's conjecture. As proved in \cite{HTGonzalez}, having bounded Scott spectral gaps implies that a class is not faithfully Borel complete, and in a way which does not require first verifying Vaught's conjecture for that class (whereas Gao's argument went through proving a strengthening of Vaught's conjecture for linear orders and trees). A central open question is whether Boolean algebras are faithfully Borel complete and whether Vaught's conjecture is true for them, and so studying Scott spectral gaps for Boolean algebras would be a way of studying the former question without having to first resolve the latter. We refer the reader to \cite{HTGonzalez} for a more detailed discussion of these connections.

In this context, it is natural to try to prove that Scott spectral gaps are bounded for trees. This is the main result of this paper.

\begin{thm}
    Given a satisfiable $\Pi_\alpha$ sentence $\varphi$ of trees, there is a tree $\mc{A} \models \varphi$ with a $\Pi_{\alpha+3}$ Scott sentence and hence Scott rank at most $\alpha + 2$.
\end{thm}

\noindent The high-level structure of the argument is similar to the case of linear orders, though of course the details differ. In particular, the reader will note that we obtain a better bound for trees than that for linear orders. While this paper is mathematically self-contained, our explanatory discussion of material that is analogous to that appearing in \cite{HTGonzalez}, especially that in Section \ref{two}, will be brief. We refer the reader to \cite{HTGonzalez} for a more detailed development.

As for linear orders, we also obtain non-trivial lower bounds, but there remains a gap between the lower bounds and the upper bounds.

\begin{thm}
    There is a $\Pi_2$ theory of trees such that no model of $\varphi$ has a $\Sigma_3$ Scott sentence, and hence every model has Scott rank at least $2$.
\end{thm}

\section{Back-and-forth relations and hierarchies of formulas}\label{two}

In this section we outline the general background needed to carry out the argument. Throughout this paper, all structures will be countable and all languages relational. Central to the ideas of the argument are the asymmetric back-and-forth relations.

\begin{defn}\label{def:bfasym}
	  The \emph{asymmetric back-and-forth relations} $\leq_\alpha$, for $\alpha < \omega_1$, are defined by:
		\begin{itemize}
			\item $(\mc{M},\bar{a}) \leq_0 (\mc{N},\bar{b})$ if $\bar{a}$ and $\bar{b}$ satisfy the same quantifier-free formulas from among the first $|\bar{a}|$-many formulas.
			\item For $\alpha > 0$, $(\mc{M},\bar{a}) \leq_\alpha (\mc{N},\bar{b})$ if for each $\beta < \alpha$ and $\bar{d} \in \mc{N}$ there is $\bar{c} \in \mc{M}$ such that $(\mc{N},\bar{b} \bar{d}) \leq_\beta (\mc{M},\bar{a} \bar{c})$.
		\end{itemize}
		We define $\bar{a} \equiv_\alpha \bar{b}$ if $\bar{a} \leq_\alpha \bar{b}$ and $\bar{b} \leq_\alpha \bar{a}$.
	\end{defn}

    \noindent The interpretation of $(\mc{M},\bar{a}) \leq_\alpha (\mc{N},\bar{b})$ is that in the back-and-forth game between $\mc{M}$ and $\mc{N}$, starting with the partial isomorphism $\bar{a} \mapsto \bar{b}$ and with the first player \textsf{Spoiler} to play next in $\mc{N}$, the second player \textsf{Duplicator} can play without losing along an ordinal clock $\alpha$. Recall that if \textsf{Duplicator} can continue playing forever, then $\mc{M} \cong \mc{N}$.

    The asymmetric back-and-forth relations can be characterized in relation to the quantifier complexity hierarchy $\Sigma_\alpha$/$\Pi_\alpha$ of $\mc{L}_{\omega_1 \omega}$-formulas.
    
    \begin{thm}[Karp \cite{Karp}]\label{thm:Karp}
	For any $\alpha \geq 1$, structures $\mc{A}$ and $\mc{B}$, and tuples $\bar{a}\in\mc{A}$ and $\bar{b}\in\mc{B}$, the following are equivalent:
	\begin{enumerate}
		\item $(\mc{A},\bar{a})\leq_\alpha (\mc{B},\bar{b})$.
		\item Every $\Pi_\alpha$ formula true about $\bar{a}$ in $\mc{A}$ is true about $\bar{b}$ in $\mc{B}$.
		\item Every $\Sigma_\alpha$ formula true about $\bar{b}$ in $\mc{B}$ is true about $\bar{a}$ in $\mc{A}$.
	\end{enumerate} 
    \end{thm} 

    \noindent However, as shown in \cite{HTGonzalez}, there is no $\Pi_\alpha$ formula defining the set of $\bar{b} \in \mc{B}$ such that $(\mc{A},\bar{a}) \leq_\alpha (\mc{B},\bar{b})$. Instead, as there, we will make use of the hierarchies $\ma_\alpha$ and $\E_\alpha$ as introduced in \cite{CGHT}.

    \begin{defn}
		All connectives below are countable.
		\begin{itemize}
			\item  $\ma_1 := \Pi_1$
			\item $\E_1 := \Sigma_1$
			\item  $\ma_\alpha :=$ closure of $\bigcup_{\beta < \alpha} \vwE_\beta$ under $\forall$ and $\bigdoublewedge$
			\item  $\E_\alpha :=$ closure of $\bigcup_{\beta < \alpha} \vwA_{\beta}$ under $\exists$ and $\bigdoublevee$
			\item  $\vwE_\alpha :=$ closure of $\E_\alpha$ under $\bigdoublevee, \bigdoublewedge$
			\item  $\vwA_\alpha :=$ closure of $\ma_\alpha$ under $\bigdoublevee, \bigdoublewedge$
		\end{itemize}
		Each $\ma_\alpha$ formula can be written in the form
		\[ \bigdoublewedge_{i \in I} \forall \bar{y}_i \varphi_i(\bar{x},\bar{y}_i)\]
		where the formulas $\varphi_i$ are $\vwE_\beta$ for some $\beta < \alpha$, and similarly a $\E_\alpha$ formula can be written in the form
		\[ \bigdoublevee_{i \in I} \exists \bar{y}_i \varphi_i(\bar{x},\bar{y}_i)\]
		where each $\varphi_i$ is $\vwA_\beta$ for some $\beta < \alpha$.
	\end{defn}
    
    \noindent The key difference between these two hierarchies is that a formula of the form $\forall \bigdoublevee \bigdoublewedge \exists$ is only $\ma_2$ while being $\Pi_4$. The two main facts we use are stated in the following lemmas.

    \begin{lem}[Chen, Gonzalez, and Harrison-Trainor \cite{CGHT}]\label{cor:karp}
	For any $\alpha \geq 1$, structures $\mc{A}$ and $\mc{B}$, and tuples $\bar{a}\in\mc{A}$ and $\bar{b}\in\mc{B}$, the following are equivalent:
	\begin{enumerate}
		\item $(\mc{A},\bar{a})\leq_\alpha (\mc{B},\bar{b})$.
		\item Every $\ma_\alpha$ formula true about $\bar{a}$ in $\mc{A}$ is true about $\bar{b}$ in $\mc{B}$.
		\item Every $\E_\alpha$ formula true about $\bar{b}$ in $\mc{B}$ is true about $\bar{a}$ in $\mc{A}$.
	\end{enumerate} 
\end{lem}

\begin{lem}[Chen, Gonzalez, and Harrison-Trainor \cite{CGHT}]\label{2.6}
    Let $\mathcal{L}$ be a countable language and $\mc{A}$ be a countable $\mathcal{L}$-structure. For each $\bar{a}\in A$ and nonzero $\alpha<\omega_1$ there is an $\ma_\alpha$ formula $\psi_{\bar{a},\mc{A},\alpha}$ such that for any countable $\mathcal{L}$-structure $\mc{B}$ and $\bar{b}\in \mc{B}$,
    \[ \mc{B} \models\psi_{\bar{a},\mc{A},\alpha}(\bar{b})\Longleftrightarrow (\mc{B},\bar{b})\geq_\alpha(\mc{A},\bar{a}).\]
\end{lem}

\section{Back-and-forth relations and trees}

A common assumption when working with back-and-forth relations and trees is that whenever a tuple $\bar{a}$ of a tree $T$ appears, it is closed under the parent relation and so is a subtree of $T$. We will make this assumption, which we will often abbreviate by saying that ``$\bar{a}$ is a finite subtree of $T$''. In the context of this paper, this assumption can be justified by the following lemma.

\begin{lem}
Let $S$ and $T$ be trees and $\bar{a} \in S$ and $\bar{b}\in T$. Let $\bar{a}' \in S$ consist of $\bar{a}$ and all of its ancestors, and similarly for $\bar{b}'$ and $\bar{b}$ in $T$, with both tuples listed in corresponding order. 
Then for any $\alpha>0$, $(S,\bar{a}) \leq_{\alpha} (T,\bar{b})$ if and only if $(S,\bar{a}') \leq_{\alpha} (T,\bar{b}')$.
\end{lem}
\begin{proof}
If $(S,\bar{a}') \leq_{\alpha} (T,\bar{b}')$ then $(S,\bar{a}) \leq_{\alpha} (T,\bar{b})$ as $\bar{a}$ and $\bar{b}$ are subtuples of $\bar{a}'$ and $\bar{b}'$ respectively. On the other hand, suppose that $(S,\bar{a}) \leq_{\alpha} (T,\bar{b})$. 
We know that if $T\models\psi(\bar{b})$ then $S \models \psi(\bar{a})$ for any $\Sigma_\alpha$ formula $\psi(\bar{x})$. In particular, if $\eta(\bar{y},\bar{b})$ is a finitary quantifier-free formula defining $\bar{b}'$ with respect to $\bar{b}$, then $S\models\eta(\bar{a}',\bar{a})$ and moreover $\bar{a}'$ is the only such tuple. We must argue that if $T\models\psi'(\bar{b}')$ then $S \models \psi'(\bar{a}')$ for any $\Sigma_\alpha$ formula $\psi'(\bar{x}')$. Suppose $T\models\psi'(\bar{b}')$. Then $T\models\exists\bar{y}\; [\eta(\bar{y},\bar{b})\land\psi'(\bar{y})]$ so $S\models\exists\bar{y} \; [\eta(\bar{y},\bar{a})\land\psi'(\bar{y})]$. Since $\eta(\bar{y},\bar{a})$ is a defining formula for $\bar{a}'$, we see $S\models\psi'(\bar{a}')$.
\end{proof}

For linear orders, the analysis in \cite{HTGonzalez} breaks the linear order up into intervals. The point was that the back-and-forth relations on linear orders can be broken up into the back-and-forth relations on intervals as follows. Given linear orders $\mc{A}$ and $\mc{B}$, and elements $a_1 < \cdots < a_n$ and $b_1 < \cdots < b_n$, we have
$(\bar{a},\mc{A}) \leq_\alpha (\bar{b},\mc{B})$ if and only if each of the corresponding intervals in $\mc{A}$ and $\mc{B}$ are related in the same way, that is $(-\infty,a_1) \leq_\alpha (-\infty,b_1)$, $(a_1,a_2) \leq_\alpha (b_1,b_2)$, and so on. We must identify an analogous way of breaking trees up into their constituent parts. We do this as follows.

\begin{defn}
    Let $\bar{a} = (a_1,\ldots,a_n)$ be a finite subtree of $T$. Define $T^{a_i}_{\bar{a}}$, the \emph{descendant tree of $a_i$ avoiding $\bar{a}$}, to be the tree consisting of $a_i$ and its descendants that are not in $\bar{a}$, with $a_i$ as its root.

       
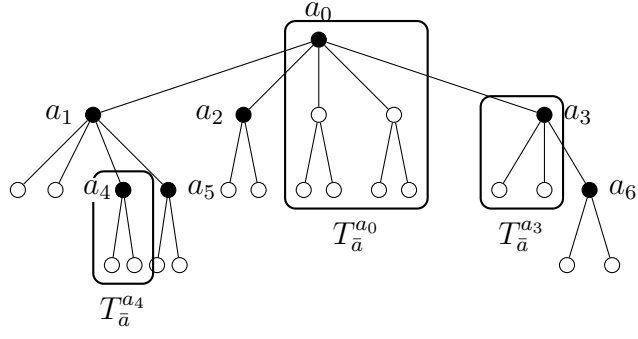
\[ \begin{tikzpicture}[blacknode/.style={circle, draw=black, fill=black,
                      inner sep=0pt, minimum size=2mm},
    whitenode/.style={circle, draw=black, fill=white,
                      inner sep=0pt, minimum size=2mm}
]

\usetikzlibrary{fit}
\pgfdeclarelayer{labels}
\pgfsetlayers{main,labels}
\node[blacknode, label=above:$a_0$] at (0,10) (a0) {};
\node[blacknode, label=left:$a_1$] at (-3,9) (a1) {};
\node[blacknode, label=left:$a_2$] at (-1,9) (a2) {};
\node[blacknode, label=right:$a_3$] at (3,9) (a3) {};

\node[blacknode, label={[fill=white, inner sep=1pt]left:$a_4$}] at (-2.6,8) (a4) {};
\node[blacknode, label=right:$a_5$] at (-2,8) (a5) {};
\node[blacknode, label=right:$a_6$] at (3.6,8) (a6) {};

\node[whitenode] at (0,9) (a0c1) {};
\node[whitenode] at (1,9)  (a0c2) {};

\node[whitenode] at (-3.5,8) (a1c1) {};
\node[whitenode] at (-4,8)   (a1c2) {};

\node[whitenode] at (-1.2,8) (a2c1) {};
\node[whitenode] at (-0.8,8)  (a2c2) {};

\node[whitenode] at (-0.2,8) (a0c1d1) {};
\node[whitenode] at (0.2,8)  (a0c1d2) {};

\node[whitenode] at (1.2,8) (a0c2d1) {};
\node[whitenode] at (0.8,8)  (a0c2d2) {};

\node[whitenode] at (2.4,8) (a3c1) {};
\node[whitenode] at (3,8) (a3c2) {};

\node[whitenode] at (-2.75,7) (a4c1) {};
\node[whitenode] at (-2.45,7) (a4c2) {};

\node[whitenode] at (-2.15,7) (a5c1) {};
\node[whitenode] at (-1.85,7) (a5c2) {};

\node[whitenode] at (3.3,7) (a6c1) {};
\node[whitenode] at (3.9,7) (a6c2) {};

\draw (a0) -- (a1);
\draw (a0) -- (a2);
\draw (a0) -- (a3);
\draw (a0) -- (a0c1);
\draw (a0) -- (a0c2);

\draw (a1) -- (a4);
\draw (a1) -- (a5);

\draw (a3) -- (a6);

\draw (a1) -- (a1c1);
\draw (a1) -- (a1c2);

\draw (a2) -- (a2c1);
\draw (a2) -- (a2c2);

\draw (a3) -- (a3c1);
\draw (a3) -- (a3c2);

\draw (a4) -- (a4c1);
\draw (a4) -- (a4c2);

\draw (a5) -- (a5c1);
\draw (a5) -- (a5c2);

\draw (a6) -- (a6c1);
\draw (a6) -- (a6c2);

\draw (a0c1) -- (a0c1d1);
\draw (a0c1) -- (a0c1d2);

\draw (a0c2) -- (a0c2d1);
\draw (a0c2) -- (a0c2d2);

\node[
    draw,
    rounded corners,
    thick,
    fit=(a0)
        (a0c1) (a0c2)
        (a0c1d1) (a0c1d2)
        (a0c2d1) (a0c2d2),
    inner sep=4pt,label=below:$T_{\bar{a}}^{a_0}$
] {};

\node[
    draw,
    rounded corners,
    thick,
    fit=(a3)
        (a3c1) (a3c2),
    inner sep=4pt,label=below:$T_{\bar{a}}^{a_3}$
] {};

\node[
    draw,
    rounded corners,
    thick,
    fit=(a4)
        (a4c1) (a4c2),
    inner sep=4pt,label=below:$T_{\bar{a}}^{a_4}$
] {};

\node[blacknode, label={[fill=white, inner sep=1pt]left:$a_4$}] at (-2.6,8) (a4) {};

\end{tikzpicture}\]
        
    Note that the trees $T^{a_i}_{\bar{a}}$ are all disjoint from each other, and that the isomorphism type of $T$ is determined by the finite tree $\bar{a}$ together with the isomorphism types of $T_{\bar{a}}^{a_i}$. $T^{a_i}_{\bar{a}}$ can be defined by the following $\Sigma_1$ formula $\eta(x;a_i,\bar{a})$:
    \[ \eta(x,a_i,\bar{a}) := \bigdoublevee_{n = 0}^\infty \exists y_0,\dots,y_n\ \left[a_i = y_0 \wedge P(y_1) = y_0 \wedge \cdots \wedge P(y_n)= y_{n-1} \wedge y_n = x \wedge \left( \bigwedge_{j \neq i}y_1 \neq a_j \right)\right].\]
\end{defn}

With the formula $\eta(x,y_i,\bar{y})$ we may relativize a formula $\varphi$ to a tree $T^{a_i}_{\bar{a}}$. We obtain a formula $\vp\restriction^{x_i}_{\bar{x}}$ of the same complexity such that if $\bar{a}\in T$ is a finite subtree of $T$ then 
\[ T\models\vp\restriction^{a_i}_{\bar{a}}\ \Longleftrightarrow T^{a_i}_{\bar{a}}\models\vp.\]
We do this by replacing each quantifier $\exists x \ldots$ in $\varphi$ with $\exists x (\eta(x,a_i,\bar{a}) \wedge \ldots)$ and each quantifier $\forall x \ldots$ in $\varphi$ with $\forall x (\eta(x,a_i,\bar{a}) \to \ldots)$. We also replace each instance of the root constant $r$ in $\varphi$ with $a_i$. Since $\eta(x,y_i,\bar{y})$ is a $\Sigma_1$ formula, $\varphi$ and $\varphi \restriction^{a_i}_{\bar{a}}$ are the same complexity $\Sigma_\alpha$/$\Pi_\alpha$ or $\me_\alpha$/$\ma_\alpha$.

\begin{remark}\label{rem:double-res}
    Corresponding to the fact that a subinterval of a subinterval of a linear order is a subinterval of that linear order, we have the following fact. Given $\bar{a}$ a subtree of $T$, and $\bar{b} = (\bar{b}_1,\ldots,\bar{b}_n)$ where each $\bar{b}_i$ is a subtree of $T^{a_i}_{\bar{a}}$ containing the roots $a_i$, note that $(T^{a_i}_{\bar{a}})^{b_{i,j}}_{\bar{b_i}}=T^{b_{i,j}}_{\bar{a}\bar{b_i}} = T^{b_{i,j}}_{\bar{b}}$. Note that each entry of $\bar{a}$ is contained in $\bar{b}$, so that for example $\bar{a}\bar{b}_i$ has $a_i$ appearing twice, which we will allow to avoid cumbersome notation. We will sometimes write $T^{b_{i,j}}_{\bar{a}\bar{b}}$ for $T^{b_{i,j}}_{\bar{b}}$.
\end{remark}

To see that the $T^{a_i}_{\bar{a}}$ are the right analogue of intervals in a linear order, we prove the following fact about how back-and-forth relations between tuples in trees break up into back-and-forth relations between their descendant trees.

\begin{lem}\label{3.1}
Let $A$ and $B$ be trees and let $\bar{a}\in A$ and $\bar{b}\in B$ be finite subtrees. Let $\alpha \geq 1$. Then $(A,\bar{a})\geq_\alpha(B,\bar{b})$ if and only if $\bar{a}\cong_{\rtree}\bar{b}$ and $A^{a_i}_{\bar{a}}\geq_\alpha B^{b_i}_{\bar{b}}$ for each $i$. Here $\bar{a}\cong_{\rtree}\bar{b}$ means $\bar{a}$ and $\bar{b}$ are isomorphic as trees.
\end{lem}
\begin{proof}
First suppose $(A,\bar{a})\geq_\alpha(B,\bar{b})$. Clearly $\bar{a}\cong_{\rtree}\bar{b}$. Given $i$, suppose that $B^{b_i}_{\bar{b}}\models\vp$ where $\vp$ is $\Pi_\alpha$. Then $B\models\vp\restriction^{b_i}_{\bar{b}}$, so $A\models\vp\restriction^{a_i}_{\bar{a}}$, meaning $A^{a_i}_{\bar{a}}\models\vp$. Thus $A^{a_i}_{\bar{a}}\geq_\alpha B^{b_i}_{\bar{b}}$.

Now suppose $\bar{a}\cong_{\rtree}\bar{b}$ and $A^{a_i}_{\bar{a}}\geq_\alpha B^{b_i}_{\bar{b}}$ for each $i$. Let $\beta<\alpha$ and $\bar{c}\in A$. Without loss of generality let $\bar{c}=(\bar{c}_1,\bar{c}_2,\ldots,\bar{c}_n)$ where $\bar{c_i}\in A^{a_i}_{\bar{a}}$ for each $i$. Then there are $\bar{d_i}\in B^{b_i}_{\bar{b}}$ such that $(A^{a_i}_{\bar{a}},\bar{c_i})\leq_\beta(B^{b_i}_{\bar{b}},\bar{d_i})$. We may assume that the $\bar{c}_i$ are subtrees of $A_{\bar{a}}^{a_i}$ containing $a_i$ and similarly for $\bar{d}_i$ and $B_{\bar{b}}^{b_i}$. Then, by the induction hypothesis we have $(A^{a_i}_{\bar{a}})^{c_{i,j}}_{\bar{c}_i} \leq_\beta (B^{b_i}_{\bar{b}})^{d_{i,j}}_{\bar{d}_i}$ which is the same as $A_{\bar{a}\bar{c}}^{c_{i,j}} \leq_\beta B^{d_{i,j}}_{\bar{b}\bar{d}}$ by Remark \ref{rem:double-res}. 
Then---again, by the induction hypothesis---we have $(A,\bar{a}\bar{c})\leq_\beta(B,\bar{b}\bar{d})$ and thus $(A,\bar{a})\geq_\alpha(B,\bar{b})$.
\end{proof}

The following lemma is a syntactic version of this previous fact. As with linear orders, the lemma is only true for $\E_\alpha$ sentences, and not for $\Sigma_\alpha$ sentences. (A counterexample arises from the sentence defined in Proposition \ref{prop:example}, with the proof being similar to the corresponding lemma from \cite{HTGonzalez}.)

\begin{lem}
\label{3.2}
Let $T$ be a tree and $\bar{a}\in T$ be a finite subtree containing the root. Suppose $T\models\vp(\bar{a})$ where $\vp$ is a $\me_\alpha$ formula. Then there are $\me_\alpha$ sentences $\theta_i$ such that $T^{a_i}_{\bar{a}}\models\theta_i$ and if $S$ is a tree and $\bar{b}\in S$ has $\bar{b}\cong_{\rtree}\bar{a}$ and $S^{b_i}_{\bar{b}}\models\theta_i$ then $S \models\vp(\bar{b})$.
\end{lem}
\begin{proof}
For $\alpha = 1$ the argument is easy (just check which subtrees $T_{\bar{a}}^{a_i}$ the witnesses to the existential formula $\varphi$ are in). Suppose $\alpha > 1$. Without loss of generality let $\vp(\bar{x})$ be of the form $\vp(\bar{x}):=\bigdoublevee_{k<\omega}\exists\bar{y} \; \psi_k(\bar{x},\bar{y})$ where each $\psi_k$ is $\bar{\ma}_\beta$ for some $\beta<\alpha$. Since $T\models\vp(\bar{a})$ there is some $\bar{c}\in T$ and $k<\omega$ such that $T\models\psi_k(\bar{a},\bar{c})$. 
Without loss of generality suppose $\bar{c}=(\bar{c}_1, \bar{c}_2,\ldots,\bar{c}_{n})$ where $\bar{c_i}$ is a subtree of $T^{a_i}_{\bar{a}}$ containing $a_i$. Note that following Remark \ref{rem:double-res} we have $(T^{a_i}_{\bar{a}})^{c_{i,j}}_{\bar{c_i}}=T^{c_{i,j}}_{\bar{c}}$. By Lemma \ref{2.6} there are $\ma_\beta$ sentences $\chi_{i,j}$ such that $S \models\chi_{i,j}$ if and only if $S \geq_{\beta}T^{c_{i,j}}_{\bar{a}\bar{c}}$. For each $i$ define
\[\theta_i:=\exists\bar{x}\ [\bar{x}\cong_{\rtree}\bar{c}_i \; \land \; \bigdoublewedge_{j}\chi_{i,j}\restriction^{x_j}_{\bar{x}}]\]
where $\bar{x}\cong_{\rtree}\bar{c_i}$ stands for the infinitary quantifier-free formula with free variables $\bar{x}$ which expresses that the parent relations on $\bar{x}$ are the same as on $\bar{c}_i$ after identifying the roots.
Then $T^{a_i}_{\bar{a}}\models\theta_i$. Suppose $S$ is a tree and $\bar{b}\in S$ is a finite subtree containing the root with $\bar{b}\cong_{\rtree}\bar{a}$ and $S^{b_i}_{\bar{b}}\models\theta_i$ for each $i$. Let $\bar{d_i}$ be the witness of $\theta_i$ in $S^{b_i}_{\bar{b}}$. Then $S^{d_{i,j}}_{\bar{b}\bar{d}}\geq_{\beta}T^{c_{i,j}}_{\bar{a}\bar{c}}$, so by Lemma \ref{3.1} $(S,\bar{d})\geq_{\beta}(T,\bar{c})$. Thus $S\models\psi_k(\bar{b},\bar{d})$ for $\bar{d} = (\bar{d}_1,\ldots,\bar{d}_n)$ and therefore $S\models\vp(\bar{b})$.
\end{proof}

\section{Proof of the main theorem}

In this section we prove the main theorem, which is that given a $\Pi_\alpha$ sentence $\varphi$ of trees, there is a tree $T \models \varphi$ such that $T$ has a $\Pi_{\alpha + 3}$ Scott sentence and hence has Scott rank at most $\alpha+2$. As in \cite{HTGonzalez}, since each $\Pi_\alpha$ sentence is $\Sigma_{\alpha+1}$ and hence $\E_{\alpha + 1}$, it will suffice to prove the stronger theorem that a satisfiable $\E_{\alpha+1}$ sentence $\varphi$ of trees has a model with a $\Pi_{\alpha+3}$ Scott sentence. After replacing $\alpha+1$ by $\alpha$ it suffices to assume that $\varphi$ is $\E_{\alpha}$ (with $\alpha \geq 2$) and show that it has a model with a $\Pi_{\alpha + 2}$ Scott sentence and hence Scott rank at most $\alpha + 1$. This sentence $\varphi$ will be fixed for the remainder of this section.

The overall strategy is as follows. We use a Henkin construction to build a model of $\varphi$ which is generic in a certain sense. Intuitively, one should think that we try to make any two tuples different from each other, if possible, by a $\me_\alpha$ sentence. If this fails, then we will try to make them be automorphic. We will show that the automorphism orbit of each tuple of this model is $\me_{\alpha+1}$-definable. This will imply that such a generic model has a $\Pi_{\alpha + 2}$ Scott sentence and hence has Scott rank at most $\alpha+1$. This last step uses a result of Chen, Gonzalez, and Harrison-Trainor \cite{CGHT} that having a $\Pi_{\alpha + 2}$ Scott sentence is equivalent to each automorphism orbit being $\me_{\alpha+1}$-definable, a strengthening (in the direction in which we apply it) of a result of Montalb\'an that having a $\Pi_{\alpha + 2}$ Scott sentence is equivalent to each automorphism orbit being $\Sigma_{\alpha+1}$-definable.


\subsection{Forcing and generic trees}

\begin{defn}
    A set $\A$ of $\mathcal{L}_{\omega_1\omega}$ formulas is an \emph{existential fragment} if it satisfies the following properties: 
    \begin{enumerate}
        \item $\A$ contains the atomic and negated atomic formulas. 
        \item $\A$ is stable under taking subformulas, finite conjunctions and disjunctions, and existential quantification.
        \item If $\vp(v_0,\dots,v_n)\in\A$ and $\tau$ is a term, then $\vp(\tau,v_1,\dots,v_n)\in\A$.
    \end{enumerate}
\end{defn}

\begin{lem}\label{4.2}
Let $\vp$ be an $\E_\alpha$ sentence of trees. Then there is a countable existential segment $\A$ of $\me_{\alpha}$ formulas which contains $\vp$. Furthermore, we may find such an $\A$ together with a countable collection $\mbC$ of countable trees  with the following properties
\begin{enumerate}
    \item For any variables $\bar{x}$ appearing in $\A$ and $\psi\in\A$, $\psi\restriction^{x_i}_{\bar{x}}\in\A$. 
    \item If $\psi(\bar{x})\in\A$ is satisfiable, then there is some $T\in\mbC$ and $\bar{a}\in T$ such that $T\models\psi(\bar{a})$.
    \item Suppose $T\in\mbC$, $\bar{a}\in T$, and $T\models\psi(\bar{a})$ where $\psi(\bar{x})\in\A$. Then there are $\theta_i\in\A$ for each $i$ as in Lemma \ref{3.2}. 
\end{enumerate}
\end{lem}
\begin{proof}
We build finite approximations $\A=\bigcup_{n<\omega}\A_n$ and $\mathbb{C}=\bigcup_{n<\omega}\mbC_n$ with $\A_n\subseteq\A_{n+1}$ and $\mbC_n\subseteq\mbC_{n+1}$ for each $n<\omega$. Set $\A_0=\{\vp\}$ and $\mbC_0=\varnothing$. We first ensure that $\A$ is an existential segment. It is easy to add all the atomic and negated atomic formulas. For each formula which appears in $\A_n$ we can ensure each of its subformulas and term substitutions appear at a later stage because there are only countably many of them. Similarly, it is easy to ensure that $\mc{A}$ satisfies the other conditions to be an existential fragment.


We satisfy the remaining three properties in a similar way. For each $\psi\in\A_k$ and variable $\bar{x}$ that appeared in $\A_k$, we can ensure $\psi\restriction^{x_i}_{\bar{x}}\in\A_n$ for some $n>k$. For infinitely many $k$, we may let $\A_{k+1}=\A_k$ and 
\[\mbC_{k+1}=\mbC_k\cup\{T_{\psi(\bar{x})} \;:\; \psi(\bar{x})\in\A_k\text{ satisfiable}\} \]
where the countable trees $T_{\psi(\bar{x})}$ are chosen such that there is $\bar{c}\in T_{\psi(\bar{x})}$ with $T_{\psi(\bar{x})}\models\psi(\bar{c}).$
Finally suppose $T\in\mbC_k$, $\psi(\bar{x})\in\A_k$, and $\bar{a}\in T$ where $T\models\psi(\bar{a})$. Then for each $i$, there is a sentence $\theta_i$ as in Lemma \ref{3.2}. We may include all of them in some $\A_n$ where $n>k$.
\end{proof}

Given the fixed $\E_\alpha$ sentence $\varphi$ we fix such $\A$ and $\mbC$ for the rest of this article. Let $\A^*$ be the satisfiable formulas of $\A$. For two sentences $\chi$ and $\psi$ in $\A^*$, we write $\chi\leq\psi$ if every model of $\chi$ is a model of $\psi$. We say $\chi\in\A^*$ \textit{forces unity} if for any $\psi_0,\psi_1\leq\chi$ with $\psi_0,\psi_1\in\A^*$, $\psi_0\land\psi_1$ is satisfiable. Also, we say $\chi\in\A^*$ \textit{forces splitting} if for every $\psi\in\A^*$ with $\psi\leq\chi$, there are $\psi_0,\psi_1\in\A^*$ with $\psi_0,\psi_1\leq\psi$ where $\psi_0\land\psi_1$ is not satisfiable. We say that $\psi_1$ and $\psi_2$ are \textit{incompatible} if they have no common extension, which is to say, $\psi_1 \wedge \psi_2$ is not satisfiable; thus $\chi$ forces splitting if densely below $\chi$ there are incompatible extensions.

We say that a tree $T$ is \textit{generic} if
\begin{enumerate}
    \item For any finite subtree $\bar{a}\in T$ containing the root there is a $\chi\in\A^*$ that forces unity or splitting such that $T\models\chi\restriction^{a_0}_{\bar{a}}$, and
    \item For any finite subtree $\bar{a}\in T$ containing the root and any $\psi\in\A^*$ either $T\models\psi\restriction^{a_0}_{\bar{a}}$ or there is some $\chi\in\A^*$ incompatible with $\psi$ such that $T\models\chi\restriction^{a_0}_{\bar{a}}$.
\end{enumerate}
We have stated all of these conditions as referring specifically to $a_0$ from among $\bar{a}$, but because they apply to all tuples from the tree in any order, they are equivalent to the conditions holding for any $a_i$, e.g., (1) is equivalent to
\begin{enumerate}
    \item[(1*)] For any finite subtree $\bar{a}\in T$ containing the root and any $i$ there is a $\chi\in\A^*$ that forces unity or splitting such that $T\models\chi\restriction^{a_i}_{\bar{a}}$.
\end{enumerate}
Writing the conditions in terms of only $a_0$ will simplify the notation, but sometimes we will apply, e.g., (1) in the form of (1*) when we have a fixed tuple $\bar{a}$ for which we want to apply the condition to every $a_i$ at the same time.

We say that a tree $T$ has property $(*)$ if for any finite subtrees $\bar{a},\bar{b} \in T$ containing the root and with $a_0 \neq b_0$ at the same level of the tree, either
\begin{enumerate}
    \item There are incompatible $\psi_0,\psi_1\in\A^*$ such that $T\models\psi_0\restriction^{a_0}_{\bar{a}}\land\psi_1\restriction^{b_0}_{\bar{b}}$ or
    \item There is some $\psi\in\A^*$ that forces unity such that $T\models\psi\restriction^{a_0}_{\bar{a}}\land\psi\restriction^{b_0}_{\bar{b}}$.
\end{enumerate}


\subsection{Construction of a generic tree}

Here we prove that there is a countable model of $\vp$ with the aforementioned properties.

\begin{thm}
Let $\varphi$ be a satisfiable $\E_\alpha$ sentence of trees. There is a countable generic tree with property $(*)$ satisfying $\vp$.
\end{thm}
\begin{proof}
We execute a Henkin-type construction. Let $C$ be a countably infinite set of constant symbols. We will build finite collections of $\me_\alpha$ sentences $\Phi_n$ in the language of trees with additional constants $C$. The $\Phi_n$ will be an increasing sequence of finite sets, starting with $\Phi_0 = \{ \varphi \}$, and each consisting of sentences of the form $\psi(\bar{c})$ where $\psi(\bar{x})\in\A^*$ and $\bar{c}\in C$. $\Phi:=\bigcup_{n<\omega}\Phi_n$ will satisfy the following conditions.
\begin{enumerate}
\item If $\bigdoublevee_{n<\omega}\psi_n\in \Phi$ then $\psi_n\in \Phi$ for some $n<\omega$.
\item If $\exists\bar{x} \; \psi(\bar{x})\in \Phi$ then $\psi(\bar{c})\in \Phi$ for some $\bar{c}\in C$.
\item If $\bigdoublewedge_{n<\omega}\psi_n\in \Phi$ then $\psi_n \in \Phi$ for every $n<\omega$.
\item If $\forall\bar{x}\; \psi(\bar{x}) \in \Phi$ then $\psi(\bar{c}) \in \Phi$ for every $\bar{c}\in C$.
\item For each atomic sentence, either it or its negation is in $\Phi$.
\item The end product is generic:
\begin{enumerate}
    \item For each $\bar{c}\in C$ such that $\Phi$ implies that $\bar{c}$ is a finite tree there is a sentence $\psi\in\A^*$ that either forces unity or splitting with $\psi\restriction^{c_0}_{\bar{c}} \in \Phi$.
    \item For each $\bar{c}\in C$ such that $\Phi$ implies that $\bar{c}$ is a finite tree, and each $\psi\in\A^*$, either $\psi\restriction^{c_0}_{\bar{c}} \in \Phi$ or there is some $\chi\in\A^*$ that's incompatible with $\psi$ and $\chi\restriction^{c_0}_{\bar{c}} \in \Phi$.
\end{enumerate}
\item The end product has the following property: for each $\bar{c}\in C$ such that $\Phi$ implies that $\bar{c}$ is a finite tree, if $\len(\bar{c})>1$ and $c_0$ and $c_1$ are at the same level, then either there are incompatible $\psi_0,\psi_1\in\A^*$ where $\psi_0\restriction^{c_0}_{\bar{c}} \in \Phi$ and $\psi_1\restriction^{c_1}_{\bar{c}} \in \Phi$ or there is a $\psi\in\A^*$ that forces unity and $\psi\restriction^{c_0}_{\bar{c}}\land\psi\restriction^{c_1}_{\bar{c}} \in \Phi$.

(This is slightly different from property ($*$) as the two elements $c_0$ and $c_1$ come from the same tuple $\bar{c}$, whereas in ($*$) the two elements $a_0$ and $b_0$ come from different tuples $\bar{a}$ and $\bar{b}$. We will show that the property above implies $(*)$. Given a pair of tuples $\bar{c},\bar{d}\in C$ such that $\Phi$ implies that $\bar{c}$ and $\bar{d}$ are finite trees, and $c_0$ and $d_0$ are at the same level, we must show that either there are incompatible $\psi_0,\psi_1\in\A^*$ where $\psi_0\restriction^{c_0}_{\bar{c}} \in \Phi$ and $\psi_1\restriction^{d_0}_{\bar{d}} \in \Phi$ or there is a $\psi\in\A^*$ that forces unity and $\psi\restriction^{c_0}_{\bar{c}}\land\psi\restriction^{d_0}_{\bar{d}} \in \Phi$. Given such $\bar{c}$ and $\bar{d}$, we can assume without loss of generality that no element of $\bar{c}$ is a descendant of $d_0$, and no element of $\bar{d}$ is a descendant of $c_0$. Thus for any tree $T$, $T_{\bar{c}\bar{d}}^{c_0} = T_{\bar{c}}^{c_0}$ and $T_{\bar{c}\bar{d}}^{d_0} = T_{\bar{d}}^{d_0}$, and similarly for a sentence $\varphi$, $\varphi\restriction_{\bar{c}\bar{d}}^{c_0} = \varphi\restriction_{\bar{c}}^{c_0}$ and $\varphi\restriction_{\bar{c}\bar{d}}^{d_0} = \varphi\restriction_{\bar{d}}^{d_0}$. We could then form the joint tuple $\bar{c}\bar{d}$, and rearrange it so that the first two entries are $c_0$ and $d_0$. We can then apply the property from (7) to this tuple. Thus the property in (7) implies ($*$).)

\end{enumerate}
In the conditions above, when we say that $\Phi$ implies that $\bar{c}$ is a finite tree, we mean that there is some $c_i$ such that $c_i = r$ is in $\Phi$, and for each other $c_j$ there is $c_k$ such that $P(c_j) = c_k$ is in $\Phi$.

We will ensure that each $\Phi_n$ is satisfiable in a tree by having, at each stage $n$, a tree $T \in \mathbb{C}$ and an assignment $v$ which assigns to each constant in $C$ an element of $T$ making $\Phi_n$ true in $T$. Moreover, for any constant $c$ appearing in $\Phi_n$ there will be a constant $d$ such that $P(c) = d$ is in $\Phi_n$ and for some $k$, $P^k(c) = r$ will be in $\Phi_n$, and we will have $c \neq d$ in $\Phi_n$ for any two constants $c,d$ in $C$. Thus $\Phi_n$ will put the structure of a finite tree containing the root on the constants that appear in $\Phi_n$.

Given $\Phi = \bigcup \Phi_n$ we can build a model of $\Phi$ whose domain is the constants $C$, as usual. This model will be a tree because of the conditions described in the previous paragraph. Since $\vp \in \Phi_0 \subseteq \Phi$ we know that $\varphi$ is true in this tree. Also, conditions (6) and (7) will guarantee that this tree is generic and has property $(*)$. 

We now describe our construction of the $\Phi_n$. We begin with $\Phi_0=\{\vp, c = r\}$ for some constant $c$ and $T\in\mbC$ satisfying $\vp$. At each stage $n+1$, we will define $\Phi_{n+1}$ by taking a single step towards satisfying a single instance of one of the above conditions. We use standard dovetailing techniques to make sure that, in the end, all of the conditions are satisfied and will just describe below how to take a single step towards satisfying each of these conditions. For some conditions, namely (3) and (4), we must handle each instance of the condition infinitely many times. As an example, in (3), if $\bigdoublewedge_{n} \psi_n \in \Phi$ then we must return at infinitely many later stages each time adding a new $\psi_n$ to $\Phi$. We cannot add them all at once as each $\Phi_n$ must be finite. For each $n$, we also attach a countable tree $T\in\mbC$ that satisfies $\Phi_n$, interpreting constants as elements from $T$. Assuming $\Phi_n$ is satisfiable, such a tree can be chosen from $\mathbb{C}$ specifically due to (2) of Lemma \ref{4.2}. 
\begin{description}
\item [Step 1:] If, at some stage $n+1$, we are handling the instance of (1) corresponding to $\bigdoublevee_{m<\omega}\psi_m \in \Phi_n$ with $T \models \Phi_n$: There is some $\psi_m$ such that $T\models\psi_m$. We add $\psi_m$ to $\Phi_{n+1}$. We still have $T \models \Phi_{n+1}$.

\item [Step 2:] If, at some stage $n+1$, we are handling the instance of (2) corresponding to $\exists x \; \psi(x) \in \Phi_n$ with $T \models \Phi_n$: There is some $a\in T$ such that $T\models\psi(a)$. If $a=c$ in $T$ for some constant symbol $c\in C$ that appears in $\Phi_n$, then add $\psi(c)$ to $\Phi_{n+1}$. Otherwise add $\psi(c')$ to $\Phi_{n+1}$ where $c'\in C$ is some constant symbol that is not mentioned in $\Phi_n$. In the latter case, we must also choose new constants for any ancestors of $a \in T$ which do not already correspond to constants, and also add to $\Phi_n$ the formulas determining the parents of any of these other new constants. We still have $T \models \Phi_{n+1}$.

\item [Step 3:] If, at some stage $n+1$, we are handling the instance of (3) corresponding to $\bigdoublewedge_{m<\omega}\psi_m \in \Phi_n$: We add $\psi_m$ to $\Phi_{n+1}$ where $\psi_m$ is the next conjunct that has not yet been added to $\Phi$. We still have $T \models \Phi_{n+1}$.

\item [Step 4:] If, at some stage $n+1$, we are handling the instance of (4) corresponding to $\forall x\;\psi(x) \in \Phi_n$: We add $\psi({c})$ to $\Phi_{n+1}$ where ${c}$ is the next constant from $C$ (according to some enumeration) for which $\psi({c})$ has not yet been added to $\Phi$. If $c$ does not appear in $\Phi_n$, we must choose some new element $a$ of $T$ for it to correspond to, and we must also choose new constants from $C$ for any ancestors of $a \in T$ which do not already correspond to constants, and also add to $\Phi_n$ the formulas determining the parents of any of these other new constants. We still have $T \models \Phi_{n+1}$.

\item [Step 5:] If at stage $n+1$ we are handling (5) and have $T \models \Phi_n$: Let $F\subseteq C$ be the finite set of constants that appears in $\Phi_n$. Then for each atomic sentence involving constants in $F$, include one of it or its negation to $\Phi_{n+1}$ based on whether or not $T$ satisfies it. We still have $T \models \Phi_{n+1}$.

\item [Step 6(a):] If at stage $n+1$ we are handling the instance of (6a) corresponding to constants $\bar{c}$ appearing in $\Phi_n$ and have $T \models \Phi_n$: Suppose that $\Phi_n$ implies that $\bar{c}$ is a finite tree. Let $\eta(\bar{c},\bar{b}):=\bigwedge \Phi_n$ where $\bar{b}$ are the constants in $\Phi_n$ that are not $\bar{c}$. Then  $T\models\exists\bar{x}\,\eta(\bar{c},\bar{x})$, so by Lemma \ref{3.2} there are sentences $\chi_i\in\A^*$ such that $T^{c_i}_{\bar{c}}\models\chi_i$ and for any countable tree $S$ if $\bar{d} \in S$ is a subtree of $S$ isomorphic to $\bar{c}$ in $T$, and if $S^{d_i}_{\bar{d}}\models\chi_i$, then $S \models\exists\bar{x}\,\eta(\bar{d},\bar{x})$. There is a sentence $\psi\leq\chi_0$ that either forces unity or splitting. Let $\Phi_{n+1}=\Phi_n\cup\{\psi\restriction^{c_0}_{\bar{c}}\}$. It remains to find a tree in $\mbC$ that satisfies $\Phi_{n+1}$. Since $\psi\in\A^*$ there is a countable tree $U \models\psi$. We modify $T$ to obtain a tree $\hat{T}$ by setting $\hat{T}_{\bar{c}}^{c_0} = U$, and $\hat{T}_{\bar{c}}^{c_i} = T_{\bar{c}}^{c_i}$ for $i > 0$.  (Essentially, we do ``surgery'' on $T$ by replacing the subtree below $c_0$ with a new subtree.) Then $\hat{T} \models \psi\restriction^{c_0}_{\bar{c}}$. Since $\hat{T}^{c_i}_{\bar{c}} \models\chi_i$ for each $i$, we have $\hat{T} \models \exists \bar{x} \eta(\bar{c},\bar{x})$ and so we can choose some assignment of the constants such that $\hat{T} \models \Phi_n$. Thus $\Phi_{n+1}$ is satisfiable, meaning there is a tree in $\mbC$ that satisfies it by (2) of Lemma \ref{4.2}.

\item [Step 6(b):] If at stage $n+1$ we are handling the instance of (6b) corresponding to constants $\bar{c}$ appearing in $\Phi_n$ and a formula $\psi \in \A^*$, and have $T \models \Phi_n$: Suppose that $\Phi_n$ implies that $\bar{c}$ is a finite tree. Set $\eta(\bar{c},\bar{b}):=\bigwedge \Phi_n$ where $\bar{b}$ are the constants in $\Phi_n$ that are not $\bar{c}$ and define the $\chi_i$ as we did in the previous step. If $\chi_0$ and $\psi$ are incompatible, we let $\Phi_{n+1}=\Phi_n\cup\{\chi_0\restriction^{c_0}_{\bar{c}}\}$ and don't change $T$. Otherwise, we let $\Phi_{n+1}=\Phi_n\cup\{\psi\restriction^{c_0}_{\bar{c}}\}$. In this case $\chi_0\land\psi$ is satisfiable, so there is a countable tree $U \models \chi_0\land\psi$. We can obtain a tree $\hat{T} \models \Phi_{n+1}$ in $\mathbb{C}$ exactly as in the previous step.

\item [Step 7:] If at stage $n+1$ we are handling the instance of (7) corresponding to constant $\bar{c}$ appearing in $\Phi_n$ with $\len(\bar{c})>1$ and $T \models \Phi_n$: Suppose that $\Phi_n$ implies that $\bar{c}$ is a finite tree. Find $\chi_i$ as before. Suppose there are incompatible $\psi_0,\psi_1\in\A^*$ where $\chi_0\land\psi_0$ and $\chi_1\land\psi_1$ are each satisfiable. Then we set $\Phi_{n+1}=\Phi_n\cup\{\psi_0\restriction^{c_0}_{\bar{c}},\psi_1\restriction^{c_1}_{\bar{c}}\}$ and find a model of $\Phi_{n+1}$ in $\mbC$ as before. Thus we may assume that there are no such $\psi_0,\psi_1$. In this case
$\chi_0\land\chi_1$ is satisfiable and forces unity. It is satisfiable because otherwise $\chi_0$ and $\chi_1$ themselves serve as the witnesses $\psi_0$ and $\psi_1$. If $\chi_0\land\chi_1$ did not force unity, then there would be incompatible $\psi_0,\psi_1\leq\chi_0\land\chi_1$ which again we assumed was not the case. So $\chi_0\land\chi_1$ is satisfiable and forces unity, and we may set $\Phi_{n+1}=\Phi_n\cup\{(\chi_0\land\chi_1)\restriction^{c_0}_{\bar{c}}\land(\chi_0\land\chi_1)\restriction^{c_1}_{\bar{c}}\}$ and find a model in $\mbC$ as before.\qedhere


\end{description}
\end{proof}

\subsection{Verification}

In this section we prove that a generic tree with property $(*)$ has Scott rank at most $\alpha+1$.

\begin{lem}
Suppose $\sigma\in\A^*$ forces unity, $A,B$ are generic countable trees, $\bar{a}\in A,\bar{b}\in B$ are finite subtrees containing the root, and $A\models~\sigma\restriction^{a_0}_{\bar{a}},B\models\sigma\restriction^{b_0}_{\bar{b}}$. Further suppose for any finite subtree $\bar{a}' \in A^{a_0}_{\bar{a}}$ containing $a_0$, and any $i$, $A^{a_{i}'}_{\bar{a}\bar{a}'}$ satisfies a sentence that forces unity. Then the same is true of every finite subtree $\bar{b}'\in B^{b_0}_{\bar{b}}$ containing~$b_0$ and every $i$, that is, $B^{b_{i}'}_{\bar{b}\bar{b}'}$ satisfies a sentence that forces unity.
\end{lem}
\begin{proof}
Assume for the sake of contradiction that there is some finite subtree $\bar{b}'\in B^{b_0}_{\bar{b}}$ containing~$b_0$ and some $i$ such that $B^{b_{i}'}_{\bar{b}\bar{b}'}$ does not satisfy a sentence that forces unity. By genericity, $B^{b_{i}'}_{\bar{b}\bar{b}'}$ satisfies a sentence that forces splitting, say $\psi$. Thus $B^{b_0}_{\bar{b}}\models\exists\bar{x}\ [ (\bar{x}\cong_{\rtree}\bar{b}') \; \land \; (\psi\restriction^{x_{i}}_{\bar{x}})]$. 
Since $A$ is generic and $A^{a_0}_{\bar{a}}$ satisfies $\sigma$ which forces unity, it also satisfies $\exists\bar{x}\ [ (\bar{x}\cong_{\rtree}\bar{b}') \land (\psi\restriction^{x_{i}}_{\bar{x}})]$. If $\bar{a}'\in A^{a_0}_{\bar{a}}$ is its witness, then $A^{a_{i}'}_{\bar{a}\bar{a}'}$ both satisfies a sentence that forces splitting and unity, a contradiction.
\end{proof}

\begin{defn}
    We say a sentence $\sigma$ is \emph{Scott-like} if it forces unity and for every generic $T\models\sigma$ and finite subtree $\bar{a}\in T$ that contains the root, $T^{a_0}_{\bar{a}}$ satisfies a sentence in $\A^*$ that forces unity. We say a tree $T$ is \emph{Scott} if it satisfies a Scott-like sentence.
\end{defn}

By the previous lemma, it does not matter which generic tree $T$ we use to check whether $\sigma$ is Scott-like.

\begin{lem}
\label{5.4.1}
    Suppose $\sigma\in\A^*$ is Scott-like, $A,B$ are generic trees, and $\bar{a}\in A,\bar{b}\in B$ are finite subtrees containing the root. Then if $A^{a_0}_{\bar{a}} \models \sigma$ and $B^{b_0}_{\bar{b}}\models\sigma$ then $A^{a_0}_{\bar{a}}\cong B^{b_0}_{\bar{b}}$.
\end{lem}
\begin{proof}
We will define a back-and-forth family $F$ consisting of pairs of tuples from $A^{a_0}_{\bar{a}}$ and from $B^{b_0}_{\bar{b}}$. We put $(\bar{u},\bar{v})\in F$ if $\bar{u}\in A^{a_0}_{\bar{a}}$ and $\bar{v}\in B^{b_0}_{\bar{b}}$ are finite subtrees containing the root, $\bar{u}\cong_{\rtree}\bar{v}$ are isomorphic as subtrees of $A_{\bar{a}}^{a_0}$ and $B_{\bar{b}}^{b_0}$, and for each $i$ both $A^{u_i}_{\bar{a}\bar{u}}$ and $B^{v_i}_{\bar{b}\bar{v}}$ satisfy the same sentence that forces unity. We prove that $F$ has the following four properties: 
\begin{enumerate}
    \item $(a_0,b_0) \in F$;
    \item if $(\bar{u},\bar{v}) \in F$ then $\bar{u} \cong_{\rtree} \bar{v}$ ;
    \item if $(\bar{u},\bar{v}) \in F$ and $u'$ is a child of some $u_i$, then there is a child $v'$ of some $v_i$ with $(\bar{u}u',\bar{v}v') \in F$;
    \item if $(\bar{u},\bar{v}) \in F$ and $v'$ is a child of some $v_i$, then there is a child $u'$ of some $u_i$ with $(\bar{u}u',\bar{v}v') \in F$;
\end{enumerate}
This is like the definition of a back-and-forth family, except that we only extend by children. Like a back-and-forth family, the existence of such an $F$ implies that $A_{\bar{a}}^{a_0}$ and $B_{\bar{b}}^{b_0}$ are isomorphic. (1) and (2) are clear.

Suppose $(\bar{u},\bar{v})\in F$ as witnessed by sentences $\varphi_i$ forcing unity with $A^{u_i}_{\bar{a}\bar{u}} \models \varphi_i$ and $B^{v_i}_{\bar{b}\bar{v}} \models \varphi_i$. Let $u'\in A$ be a child of $u_i$. Since $A_{\bar{a}}^{a_0} \models \sigma$ and $\sigma$ is Scott-like, there are sentences $\psi$ and $\psi'$ that force unity and $A^{u_i}_{\bar{a}\bar{u}u'}\models\psi$ and $A^{u'}_{\bar{a}\bar{u}u'}\models\psi'$. Note that $A_{\bar{a}\bar{u}u'}^{u_j} = A_{\bar{a}\bar{u}}^{u_j}$ for $j \neq i$, and so still satisfies $\varphi_j$. Define
\[\eta:=\exists x \ \left[ P(x) = r \;\land \; \psi \restriction^r_{rx} \;\land\; \psi' \restriction^x_{rx} \right].\]
Then $A^{u_i}_{\bar{a}\bar{u}}\models\eta$ and thus $\vp_i\land\eta$ is satisfiable. If $B^{v_i}_{\bar{b}\bar{v}}\nvDash\eta$ then by genericity, there is some sentence $\theta$ such that $B^{v_i}_{\bar{b}\bar{v}}\models\theta$ but $\theta$ and $\eta$ are incompatible ($\theta \wedge \eta$ is not satisfiable). This contradicts that $\vp_i$ forces unity, and thus $B^{v_i}_{\bar{b}\bar{v}}\models\eta$. If $v'$ is the witness of $B^{v_i}_{\bar{b}\bar{v}}$ to $\eta$, then $v'$ is a child of $v_i$, $B^{v_i}_{\bar{b}\bar{v}v'} \models \psi$, and $B^{v'}_{\bar{b}\bar{v}v'} \models \psi'$ (while we still have, for $j \neq i$, $B^{v_j}_{\bar{b}\bar{v}v'} = B^{v_j}_{\bar{b}\bar{v}} \models \vp_j$) . Thus $(\bar{u}u',\bar{v}v')\in F$. The other direction follows similarly.
\end{proof}

\begin{lem}\label{lem:final}
Each automorphism orbit of a generic tree with property $(*)$ is definable by an $\me_{\alpha+1}$ formula. In particular, such a tree has a $\Pi_{\alpha+2}$ Scott sentence and hence has Scott rank at most $\alpha+1$.
\end{lem}
\begin{proof}
Let $T$ be a generic tree with property ($*$). If we can show that each automorphism orbit of $T$ is definable by an $\me_{\alpha+1}$ formula, then by Theorem 7.7 of \cite{CGHT}, $T$ has a $\Pi_{\alpha+2}$ Scott sentence.

Let $\bar{a}=(a_0,\dots,a_{n-1})\in T$ be a finite subtree containing the root. For each $i<n$ we define $\vp_i$ the following way. If $T^{a_i}_{\bar{a}}\models\psi$ where $\psi$ is Scott-like, then set $\vp_i:=\psi$. Otherwise, there's some $\bar{b}\in T^{a_i}_{\bar{a}}$ where $T^{b_0}_{\bar{a}\bar{b}}\models\psi$ for some $\psi$ that forces splitting. By property $(*)$, for each finite tree $\bar{c}\in T$ where $c_0 \neq b_0$ is of the same distance from the root as $b_0$, there is some $\psi_{\bar{c}}$ such that $T^{b_0}_{\bar{a}\bar{b}}\models\neg\psi_{\bar{c}}$ but $T^{c_0}_{\bar{c}}\models\psi_{\bar{c}}$. Using these, set
\[\vp_i:=\exists\bar{y} \; \left[ \left(\bar{y}\cong_{\rtree}\bar{b} \right) \; \land \; \bigdoublewedge_{\bar{c}}\neg\psi_{\bar{c}}\restriction^{y_0}_{\bar{y}}\right].\]
For each $i$ we have $T_{\bar{a}}^{a_i} \models \varphi_i$.

Now define $\theta(\bar{x})$ by
\[\theta(\bar{x}):=\left(\bar{x}\cong_{\rtree}\bar{a}\right) \;\land \; \bigwedge_{i<n}\vp_i\restriction^{x_i}_{\bar{x}}.\]
Since each $\vp_i$ is $\me_{\alpha+1}$, $\theta(\bar{x})$ is $\me_{\alpha+1}$. Clearly $T\models\theta(\bar{a})$. We prove that $\theta(\bar{x})$ defines the automorphism orbit of $\bar{a}.$ Suppose $T\models\theta(\bar{d})$ for $\bar{d}\in T$, and we want to show that $(T,\bar{a}) \cong (T,\bar{d})$. We have, for each $i$, $T_{\bar{d}}^{d_i} \models \varphi_i$. Since we already have $\bar{d}\cong_{\rtree}\bar{a}$, it remains to prove $T^{d_i}_{\bar{d}}\cong T^{a_i}_{\bar{a}}$ for each $i<n$. We induct on the tree structure of $\bar{a}$, starting with the leaves as the base case. Suppose $a_i$ is a leaf of $\bar{a}$. If $\vp_i$ is Scott-like, $T^{d_i}_{\bar{d}}\cong T^{a_i}_{\bar{a}}$ by Lemma \ref{5.4.1}. If $\vp_i$ is not Scott-like, if $a_i = d_i$ then since they are leaves we have $T_{\bar{a}}^{a_i} = T_{\bar{d}}^{d_i}$. We argue that $a_i = d_i$ by contradiction; suppose to the contrary that $a_i \neq d_i$. For $\bar{b}$ the tuple in $T_{\bar{a}}^{a_i}$ used in the definition of $\varphi_i$, we have $T^{d_i}_{\bar{d}}\models\exists \bar{y} \; [(\bar{y}\cong_{\rtree}\bar{b}) \; \land \; \bigdoublewedge_{\bar{c}}\neg\psi_{\bar{c}}\restriction^{y_0}_{\bar{y}}]$. Since $a_i$ and $d_i$ are of the same distance from the root and $a_i\neq d_i$, for any possible witness $\bar{y} = \bar{c}$ in $T^{d_i}_{\bar{d}}$ we would have $b_0 \neq c_0$ as they extend $a_i$ and $d_i$ respectively. Also, $b_0$ and $c_0$ are of the same height. Thus $T_{\bar{d}}^{d_i} \models \psi_{\bar{c}}\restriction^{c_0}_{\bar{c}}$, contradicting that $\bar{y} = \bar{c}$ was a witness to $T_{\bar{d}}^{d_i} \models \varphi_i$.

Now for the inductive case, suppose $T^{d_j}_{\bar{d}}\cong T^{a_j}_{\bar{a}}$ for every descendant $a_j$ of $a_i$ and $d_j$ of $d_i$. We must argue that $T_{\bar{a}}^{a_i} \cong T_{\bar{d}}^{d_i}$. If $\vp_i$ is Scott-like, $T^{d_i}_{\bar{d}}\cong T^{a_i}_{\bar{a}}$. If not, by similar reasoning as before, $a_i=d_i$. This part of the argument in the previous paragraph did not use the assumption that $a_i$ and $d_i$ were leaves; it was only to then argue that if $a_i = d_i$ then $T^{d_i}_{\bar{d}}\cong T^{a_i}_{\bar{a}}$ that we used this fact. We will now argue for this same fact but this time we must use the inductive hypothesis. Consider the children of $a_i = d_i$. Without loss  of generality let the children of $a_i$ which also appear in $\bar{a}$ be $a_0,\dots,a_k$. Similarly, the children of $d_i$ which also appear in $\bar{d}$ are $d_0,\dots,d_k$. Using the induction hypothesis, the full subtrees of $T$ below each of $a_0,\dots,a_k$ are isomorphic to the full subtrees of $T$ below each of $d_0,\ldots,d_k$ respectively. The isomorphism type of $T_{\bar{a}}^{a_i}$ is determined by the isomorphism types of the full subtrees of $T$ below the children of $a_i$ which do not appear in $\bar{a}$, and the isomorphism type of $T_{\bar{d}}^{d_i}$ is determined by the isomorphism types of the full subtrees of $T$ below the children of $d_i$ which do not appear in $\bar{d}$. Since $a_i = d_i$, the isomorphism types of full subtrees of $T$ below the children of $a_i$ are the same as the isomorphism types of full subtrees of $T$ below the children of $d_i$, counting multiplicity. Taking away the isomorphism types corresponding to $a_0,\ldots,a_k$ on one side and $d_0,\ldots,d_k$ on the other, we still have the same isomorphism types. Thus, we conclude that $T_{\bar{a}}^{a_i} \cong T_{\bar{d}}^{d_i}$.
\end{proof}

\begin{remark}\label{rem:diff}
    A key fact for the analogue of Lemma \ref{lem:final} for linear orders, in the treatment of overlapping intervals, was a theorem of Lindenbaum and Tarski \cite{LindenbaumTarski}: If $K$ and $L$ are linear orders, and for each $n$ the product order $K \cdot n$ is an initial segment of $L$, then $K \cdot \omega$ is an initial segment of $L$ and so $K + L \cong L$. No analogous fact is required for trees. Instead, we are able to argue that we do not need to consider overlapping trees. This is responsible for the improved upper bounds we obtain.
\end{remark}

\section{Lower bounds}

Here we give a lower bound: We find a $\Pi_2$ theory of trees with no models with a $\Sigma_3$ Scott sentence. 

\begin{prop}\label{prop:example}
There is a satisfiable $\Pi_2$ sentence $\vp$ extending the axioms of trees such that every model of $\vp$ has a $\Pi_3$ Scott sentence but no $\Sigma_3$ Scott sentence. 
\end{prop}
\begin{proof}
Consider a language with the parent relation $P$, constant symbol $r$ for the root, and unary predicates $R_i$ for each $i<\omega$. One may think of the $R_i$ as colours, and we call a tree with the additional relations $R_i$ a \textit{coloured tree}. We can transform any coloured tree into a tree, and these two structures are effectively $\Sigma_1$ bi-interpretable with the bi-interpretation being uniform among all coloured trees. (See \cite{MBook} or \cite{HTMMM,HTMM} for background on infinitary bi-interpretability.) Thus we can prove the theorem by constructing such an example of coloured trees. The transformation can be defined recursively. Given a coloured tree $T$ with root $r$, let $a_0,a_1,\ldots$ be the children of $r$, and let $T_i$ be the coloured subtrees of $T$ below $a_i$.
\begin{center}
    \begin{tikzpicture}[baseline=(current bounding box.center),level 2/.style={sibling distance=18mm},level 1/.style={sibling distance=30mm}]
  \node {$r$} [grow'=up]
    child {node {$a_0$}
      child {node (c1) {}}
      child {node (c2) {}}
    }
    child {node {$a_1$}
      child {node (c3) {}}
      child {node (c4) {}}
    }
    child {node {$a_2$}
      child {node (c5) {}}
      child {node (c6) {}}
    }
    child {node {$\cdots$}
};

    \node at ($(c1)!0.5!(c2)$) {$T_0$};
    \node at ($(c3)!0.5!(c4)$) {$T_1$};
    \node at ($(c5)!0.5!(c6)$) {$T_2$};
\end{tikzpicture}
.
\end{center}
Then $\Phi(T)$ will be the following tree.
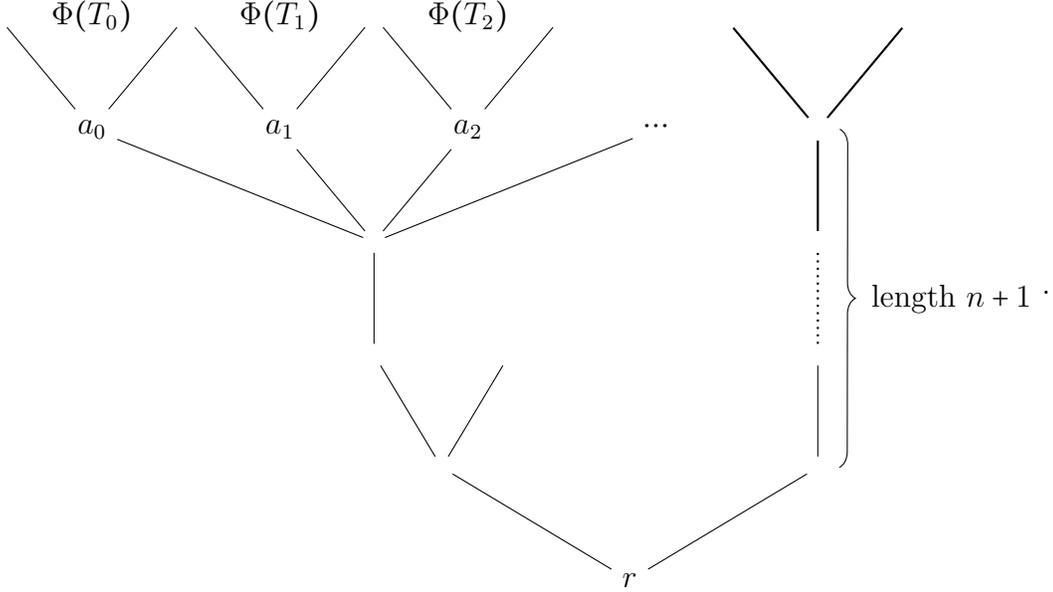
\begin{center}\begin{tikzpicture}[baseline=(current bounding box.center),level 2/.style={sibling distance=18mm},level 1/.style={sibling distance=50mm},level 3/.style={sibling distance=25mm},level 4/.style={sibling distance=25mm}]
  \node {$r$} [grow'=up]
    child {node {}
    child {node {}
    child {node {}
    child {node {$a_0$}
      child {node (c1) {}}
      child {node (c2) {}}
    }
    child {node {$a_1$}
      child {node (c3) {}}
      child {node (c4) {}}
    }
    child {node {$a_2$}
      child {node (c5) {}}
      child {node (c6) {}}
    }
    child {node {$\cdots$}
}}}child {node {}}}
    child {node (c0) {}
        child {node {}
        child {node {} edge from parent[dotted, thick]
        child {node (cn) {} edge from parent[solid]
        child {node {}}
        child {node {} child {node {}}}}}
        }
    };

    \node at ($(c0)+(42mm,0)$) {$\cdots$};

    \node at ($(c1)!0.5!(c2)$) {$\Phi(T_0)$};
    \node at ($(c3)!0.5!(c4)$) {$\Phi(T_1)$};
    \node at ($(c5)!0.5!(c6)$) {$\Phi(T_2)$};

    \draw[
  decorate,
  decoration={brace, mirror, amplitude=6pt}
]
  ($(c0.east)+(4pt,0)$) --
  ($(cn.east)+(4pt,0)$)
  node[midway, right=8pt] {length $2n+1/2n+2$};
\end{tikzpicture}
.\end{center}
To the left we recursively apply this procedure to the subtrees of the root $r$. On the right, we add a widget of length $2n+1$ when $r$ has colour $n$, and $2n+2$ when $r$ does not have colour $n$. In the diagram, we only show one such widget, but there will be many of them. It is straightforward to see that each of the colours is both universally and existentially definable in $\Phi(T)$, and $T$ and $\Phi(T)$ are infinitary bi-interpretable. For the remainder of the proof we work with coloured trees.

The sentence $\varphi$ will say that $P$ is the parent relation of a tree with root $r$. The sentence $\varphi$ will also say that the tree has height 1 ($\forall x \; x = r \vee P(x) = r$) and is infinite ($\exists^\infty x \; P(x) = r$), as well as that any two children of the root differ on some colour ($\forall x\neq y\bigdoublevee_{n<\omega}\neg(R_n(x)\leftrightarrow R_n(y))$) and that each finite set of colours is realized. This last condition can be expressed as \[\bigdoublewedge_{\sigma\in2^{<\omega}}\exists x\bigwedge_{i<\len(\sigma)}(\neg)^{\sigma(i)}R_i(x)\]
where $(\neg)^{1}$ is just $\neg$ and $(\neg)^{0}$ is ignored. Thus $\vp$ is $\Pi_2$ and one can show it is consistent using standard arguments. 

By [Mon15] it suffices to show that if $T\models\vp$ then each automorphism orbit is $\Sigma_2$ definable but there is some automorphism orbit that is not $\Sigma_1$ definable, even with parameters. Let $\bar{a}\in T$. For $i<\len(\bar{a})$ let $\sigma_i\in2^\omega$ be such that $T\models\bigdoublewedge_{n<\omega}(\neg)^{\sigma_i(n)}R_n(a_i)$. Then the formula $\bigwedge_{i<\len(\bar{a})}\bigdoublewedge_{n<\omega}(\neg)^{\sigma_i(n)}R_n(x_i)$ defines $\bar{a}$, and this is a $\Pi_1$ formula. 

For the sake of contradiction suppose every automorphism orbit is $\Sigma_1$-definable with parameters $\bar{p}\in T$. Since $T$ is infinite, we may pick some $b\in T$ that is not the root and not part of $\bar{p}$. Suppose $\exists\bar{y}\,\psi(x,\bar{y},\bar{p})$ is a $\Sigma_1$ formula defining $b$, where $\psi$ is a first order quantifier free formula. Let $T\models\psi(b,\bar{c},\bar{p})$ where $\bar{c}\in T$ and let $R_0,\dots,R_{n-1}$ be the only unary predicates that appear in $\psi$. Because every finite choice of colours appears and every vertex receives different colours, there is a $b'\in T$ that is not $b$, the root, nor one of $\bar{p}$, but agrees with $b$ on the first $n$ colours, i.e., $T\models\bigdoublewedge_{i<n}R_i(b')\leftrightarrow R_i(b)$. If $\bar{c}'$ is $\bar{c}$ with changing the appearances of $b$ to $b'$, then $(b,\bar{c},\bar{p})$ and $(b',\bar{c}',\bar{p})$ have the same first order quantifier free diagram if we restrict the colours to $R_0,\dots,R_{n-1}$. Thus $T\models \psi(b',\bar{c}',\bar{p})$, meaning $b\cong b'$. However, since $b$ and $b'$ disagree on one of the colours $R_i$ (for $i\geq n$), this is a contradiction.
\end{proof}

\bibliographystyle{alpha}
\bibliography{References}

\end{document}